\documentclass[a4paper]{amsart}

\usepackage{amssymb, enumitem}
\usepackage[all]{xy}
\usepackage{hyperref, aliascnt, graphicx}

\newcommand{\andSep}{\,\,\,\text{ and }\,\,\,}
\newcommand{\ZZ}{{\mathbb{Z}}}
\newcommand{\NN}{{\mathbb{N}}}
\newcommand{\CC}{{\mathbb{C}}}
\newcommand{\dist}{{\mathrm{dist}}}
\newcommand{\ca}{C*-algebra}

\makeatletter
\newcommand*\bigcdot{\mathpalette\bigcdot@{.75}}
\newcommand*\bigcdot@[2]{\mathbin{\vcenter{\hbox{\scalebox{#2}{$\m@th#1\bullet$}}}}}
\makeatother

\newcommand{\multNoSpan}{\bigcdot}

\newenvironment{psmallmatrix}
  {\left(\begin{smallmatrix}}
  {\end{smallmatrix}\right)}

\def\today{\number\day\space\ifcase\month\or   January\or February\or
   March\or April\or May\or June\or   July\or August\or September\or
   October\or November\or December\fi\   \number\year}

\newtheorem{lma}{Lemma}[section]

\newaliascnt{thmCt}{lma}
\newtheorem{thm}[thmCt]{Theorem}
\aliascntresetthe{thmCt}

\newaliascnt{corCt}{lma}
\newtheorem{cor}[corCt]{Corollary}
\aliascntresetthe{corCt}

\newaliascnt{prpCt}{lma}
\newtheorem{prp}[prpCt]{Proposition}
\aliascntresetthe{prpCt}

\theoremstyle{definition}

\newaliascnt{dfnCt}{lma}
\newtheorem{dfn}[dfnCt]{Definition}
\aliascntresetthe{dfnCt}

\newaliascnt{rmkCt}{lma}
\newtheorem{rmk}[rmkCt]{Remark}
\aliascntresetthe{rmkCt}

\newaliascnt{rmksCt}{lma}
\newtheorem{rmks}[rmksCt]{Remarks}
\aliascntresetthe{rmksCt}

\newaliascnt{exaCt}{lma}
\newtheorem{exa}[exaCt]{Example}
\aliascntresetthe{exaCt}

\newaliascnt{qstCt}{lma}
\newtheorem{qst}[qstCt]{Question}
\aliascntresetthe{qstCt}

\newcounter{theoremintro}

\newaliascnt{thmIntroCt}{theoremintro}
\newtheorem{thmIntro}[thmIntroCt]{Theorem}
\aliascntresetthe{thmIntroCt}

\title{Rings and C*-algebras generated by commutators}

\author[Eusebio Gardella]{Eusebio Gardella}
\address{Eusebio Gardella
Department of Mathematical Sciences, Chalmers University of
Technology and University of Gothenburg, Gothenburg SE-412 96, Sweden.}
\email{gardella@chalmers.se}
\urladdr{www.math.chalmers.se/~gardella}

\author{Hannes Thiel}
\address{Hannes~Thiel, 
Department of Mathematical Sciences, Chalmers University of Technology and University of
Gothenburg, Gothenburg SE-412 96, Sweden.}
\email{hannes.thiel@chalmers.se}
\urladdr{www.hannesthiel.org}

\thanks{
The first named author was partially supported by the Swedish Research Council Grant 2021-04561.
The second named author was partially supported by the Knut and Alice Wallenberg Foundation (KAW 2021.0140).
}
 	
\subjclass[2010]%
{Primary
16N60, 
46L05, 
47B47. 
Secondary
15A15, 
15A23, 
16S50, 
19K14. 
}
\date{\today}

\begin{document}

\begin{abstract}
We show that a unital ring is generated by its commutators as an ideal if and only if there exists a natural number $N$ such that every element is a sum of $N$ products of pairs of commutators.
We show that one can take $N \leq 2$ for matrix rings, and that one may choose $N \leq 3$ for rings that contain a direct sum of matrix rings -- this in particular applies to \ca{s} that are properly infinite or have real rank zero.
For Jiang-Su-stable \ca{s}, we show that $N\leq 6$ can be arranged.

For arbitrary rings, we show that every element in the commutator ideal admits a power that is a sum of products of commutators.
Using that a \ca{} cannot be a radical extension over a proper ideal, we deduce that a \ca{} is generated by its commutators as a not necessarily closed ideal if and only if every element is a finite sum of products of pairs of commutators.
\end{abstract}

\maketitle

\section{Introduction}

The commutator of two elements $x$ and $y$ in a ring is defined as $[x,y] := xy - yx$, and it is an elementary fact that the ideal generated by all commutators is the smallest ideal whose associated quotient ring is commutative.
We study rings that are `very noncommutative', in the sense that they are generated by their commutators as an ideal; 
equivalently, these are the rings that admit no nonzero ring homomorphism
to a commutative ring.

A related property that has been studied is that of being \emph{additively} generated by commutators \cite{Har58CommutatorDivRg, Mes06CommutatorRings}.
This is however a much more restrictive condition, as for example the matrix algebras $M_n(\ZZ)$ for $n \geq 2$ do not satisfy it, while they are easily seen to  be generated by their commutators as an ideal.

For a unital ring, we show in \autoref{prp:UnitalAlgGenCommut} 
that being generated by its
commutators as an ideal implies that every element is a sum of products of \emph{pairs} of commutators, and there even is a uniform bound (depending only on the ring) on the number of summands. 
We reproduce a particular case here:

\begin{thmIntro}
\label{thmA}
Let $R$ be a unital ring.
Then the following are equivalent:
\begin{enumerate}
\item
The ring $R$ is generated by its commutators as an ideal.
\item
The ring $R$ is generated by its commutators as a ring.
\item
There exists $N \in \NN$ such that for every $a \in R$ there exist $b_j,c_j,d_j,e_j \in R$, for $j=1,\ldots,N$, such that
\[
a = \sum_{j=1}^N [b_j,c_j][d_j,e_j].
\]
\end{enumerate}
\end{thmIntro}

The fact that~(1) implies~(3) in the above theorem is most surprising, for two reasons: 
First, it states that double-products are sufficient to generate $R$. 
And second, it also shows that there exists a uniform bound on the number of summands needed. 

A result of Baxter \cite[Theorem~1]{Bax65CommutatorSubgroupRing} asserts that a simple ring is either a field, or every of its elements is a sum of products of pairs of commutators.
\autoref{thmA} recovers Baxter's result in the unital case, and moreover shows that the minimal number of summands required is uniformly bounded.

\medskip

\autoref{thmA} can be viewed as a Waring-type result, as we explain now. 
The classical Waring problem asks if for each $k \geq 1$ there exists a uniform bound $g(k)$ such that every integer can be written as a sum of at most $g(k)$ many $k$-th powers of integers.
This was solved positively by Hilbert in 1909, and since then numerous generalizations and variations of this problem have been considered.
In the context of general rings, the Waring problem was introduced by Bre\v{s}ar in \cite{Bre20CommutatorsImageNCPoly}; see also \cite{Lee22AddSubgpGenNCPoly, BreSem23WaringPbmMatrixAlgs, BreSem23WaringPbmMatrixAlgs2}.
Here, given a unital ring $R$ and a polynomial $f \in \ZZ\langle x_1,\ldots,x_d \rangle$ in $d$ noncommuting variables, one considers the set $f(R) := \{ f(a_1,\ldots,a_d)\colon a_1,\ldots,a_d \in R\}$ and asks if there exists an integer $N$ such that every element in $R$ is a sum of at most $N$ elements in $f(R)$.
The classical Waring problem is the case $R=\ZZ$ and $f(x)=x^k$.
Our results apply to the polynomial in four variables given by
\[
f(a,b,c,d) 
= [a,b][c,d]
= abcd - bacd -abdc + badc.
\]
\autoref{thmA} shows that if a unital ring $R$ is generated by $f(R)$ then there is a constant $N$ (depending only on $R$) such that every element is a sum of at most $N$ elements in $f(R)$.

The question of when a unital ring is generated by its commutators has also been studied in 
\cite{Ero22SubrgGenCommutators}.

\medskip

For not necessarily unital rings that are generated by their commutators as an ideal, we show in \autoref{prp:NonunitalAlgGenCommut} that sums of products of commutators contain an ideal over which the ring is a radical extension.
Here, a ring $R$ is said to be a radical extension over a subring $S \subseteq R$ if for every $x \in R$ there exists $n \geq 1$ with $x^n \in S$. 
The following is a particular case of \autoref{prp:NonunitalAlgGenCommut}.

\begin{thmIntro}
\label{thmB}
Let $R$ be a (not necessarily unital) ring that is generated by its commutators as an ideal.
Then for every $a \in R$, there exist $n,m \geq 1$ and $b_j,c_j,d_j,e_j \in R$ for $j = 1,\ldots,n$ such that
\[
a^m = \sum_{j=1}^n [b_j,c_j][d_j,e_j].
\]
\end{thmIntro}

Of course, unital rings are never radical extensions over proper ideals.
By \cite[Proposition~2.2]{GarThi23arX:PrimeIdealsCAlg}, \ca{s} have this property as well.
Using these facts, we obtain the following characterization of when a \ca{} is generated by its commutators; see \autoref{prp:CommutatorsCAlg}.

\begin{thmIntro}
\label{thmC}
Let $A$ be a \ca.
Then the following are equivalent:
\begin{enumerate}
\item
The \ca{} $A$ is generated by its commutators as an ideal.
\item
The \ca{} $A$ is generated by its commutators as a ring.
\item
For every $a \in A$ there exist $n \in \NN$ and $b_j,c_j,d_j,e_j \in A$ for $j = 1,\ldots,n$ such that
\[
a = \sum_{j=1}^n [b_j,c_j][d_j,e_j].
\]
\end{enumerate}
\end{thmIntro}

If a \ca{} is generated by its commutators as an ideal, then it admits no one-dimensional irreducible representations.
For unital \ca{s}, the converse also holds, but there exist nonunital \ca{s} that admit no one-dimensional irreducible representations which are not generated by their commutators as an ideal;
see \autoref{exa:NoCharNotGenCommutators}.

We note that, unless $A$ is unital, it is not clear if there exists a uniform bound on the number of summands in part~(3) of \autoref{thmC};
see \autoref{qst:BoundNonunitalCAlg}.

Robert showed in \cite[Theorem~3.2]{Rob16LieIdeals} that if a unital \ca{} $A$ admits no one-dimensional irreducible representations, then every element is a sum of commutators with a sum of products of commutators.
\autoref{thmA} strengthens Robert's result by showing that is suffices to consider summands that are products of commutators;
see \autoref{rmk:ReplaceCommutatorsByTheirProducts}.
We point out that it is not clear that every commutator in a \ca{} (let alone in a ring) is a sum of products of pairs of commutators;
see \autoref{qst:Commutators}.
This and related questions are discussed in \autoref{sec:Outlook}.

Our proofs rely on the theory of Lie ideals in associative rings, as developed by Herstein \cite{Her70LieStructure} and others in the 1950s and 1960s.
In the context of \ca{s}, Lie ideals have been studied by Miers \cite{Mie81ClosedLie} and coauthors in the 1970s and 1980s, and more recently by Bre\v{s}ar, Kissin, and Shulman \cite{BreKisShu08LieIdeals}, and Robert \cite{Rob16LieIdeals}.

\medskip

In \autoref{sec:BoundsRing} we study the invariant $\xi(R)$ that associates to a unital ring $R$ generated by its commutators, the minimal $N$ for which part~(3) of \autoref{thmA} holds;
see \autoref{dfn:Xi}.
We obtain estimates for this invariant for matrix rings over arbitrary unital rings (\autoref{prp:MatrixSizeN}), for division rings (\autoref{exa:DivisionRing}), for semisimple, Artinian rings (\autoref{exa:Semisimple}), as well as
for unital rings that contain a direct sum of matrix rings (\autoref{prp:XiMatrixSubalgebra}).

In \autoref{sec:BoundsCAlg}, we specialize to unital \ca{s} and we obtain estimates for the invariant $\xi$ for \ca{s} that are properly infinite (\autoref{exa:ProperlyInfinite}), have real rank zero (\autoref{prp:RR0}), or contain a unital copy of the Jiang-Su algebra $\mathcal{Z}$ (\autoref{prp:DimensionDrop}). 
The latter in particular applies to all \ca{s} covered by the Elliott classification programme,
as well as to the reduced group \ca{s} $C^*_\lambda(\mathbb{F}_n)$ of all nonabelian free groups. 

In forthcoming work \cite{GarThi24pre:ProdCommutatorsVNA}, we show that $\xi(M) \leq 2$ for von Neumann algebras~$M$ that have no commutative summand, and that many such von Neumann algebras even satisfy the optimal value $\xi(M)=1$.
In \cite{BreGarThi24pre:ProdCommutatorsMatrixRgs}, we show that $\xi(M_n(D))=1$ for every division ring $D$ with infinite center and $n \geq 2$, and we give an example of a commutative, unital ring $R$ such that $\xi(M_2(R)) = 2$.

\medskip

We have written this paper with both ring theorists and operator algebraists in mind. 
To make it accessible to both communities, we have included some basic definitions and statements of useful results along the way. 
We also sometimes mention $L^p$-operator algebras, for which we refer the reader to \cite{Gar21ModernLp}  for an introduction.

\subsection*{Acknowledgements}

The authors thank Matej Bre\v{s}ar, Tsiu-Kwen Lee and Leonel Robert for valuable comments on earlier versions of this paper.

\section{A commutativity result}
\label{sec:Commutativity}

In this section, we study rings $R$ satisfying $\big[[R,R],[R,R]^2\big]=\{0\}$, that is, rings where commutators commute with products of pairs of commutators.
We show that a semiprime ring with this property is automatically commutative;
see \autoref{prp:CommutativitySemiprime}.
For general rings, said condition implies that the commutator ideal is nil, meaning that each of its elemens is nilpotent;
see \autoref{prp:CommutatorIdealNil}.

\medskip

Given elements $x,y$ in a ring, we use $[x,y] := xy - yx$ to denote the (additive) commutator.
Following the customary convention in algebra, when $X$ and $Y$ are subsets of a ring, we use $[X,Y]$ to denote the additive subgroup generated by the set of commutators $[x,y]$ with $x \in X$ and $y \in Y$.
We sometimes consider just the set of commutators, and then we use the notation
\[
[X,Y]_1 := \big\{ [x,y]\colon x \in X, y \in Y \big\}
\]
for given subsets $X$ and $Y$ of a ring.

Similarly, we use $XY$ to denote the additive subgroup generated by the set $\{xy\colon x \in X, y \in Y \}$, and we write $X^2$ for $XX$.
If we want to specify the set of products, we use the notation
\[
X \multNoSpan Y := \big\{ xy \colon x\in X, y\in Y \big\}
\]
for given subsets $X$ and $Y$ of a ring. Note that $XY$ is therefore the 
additive subgroup generated by $X\multNoSpan Y$.

An ideal $I\subseteq R$ in a ring is said to be \emph{prime} if $I \neq R$ and whenever $J_1,J_2 \subseteq R$ are ideals with $J_1 J_2 \subseteq I$, then we have $J_1 \subseteq I$ or $J_2 \subseteq I$.
A ring $R$ is \emph{prime} if~$\{0\}$ is a prime ideal, that is, whenever $I,J\subseteq R$ are ideals and $I J=\{0\}$, then we have $I=\{0\}$ or $J=\{0\}$. 
(Equivalently, whenever $a,b\in R$ satisfy $aRb=\{0\}$, then either $a=0$ or $b=0$.)
Further, a ring $R$ is \emph{semiprime} if the intersection of all prime ideals in $R$ is $\{0\}$.
(Equivalenty, whenever $a \in R$ satisfies $aRa = \{0\}$, then $a=0$.)
We refer to \cite[Section~10]{Lam01FirstCourse2ed} for details.

An additive subgroup $L$ of a ring $R$ is said to be a \emph{Lie ideal} if $[R,L] \subseteq L$, or equivalently $[L,R] \subseteq L$.
Finally,
recall that a ring $R$ is said to have no $2$-torsion if $x\in R$ and 
$2x=0$ imply $x=0$.

\begin{prp}
\label{prp:CommutativitySemiprimeNo2Torsion}
Let $R$ be a ring, and let $L \subseteq R$ be a Lie ideal.
Consider the following conditions:
\begin{enumerate}
\item
We have $[L,R]=\{0\}$, that is, $L$ is a subset of the center of $R$.
\item
We have $[L,L]=\{0\}$.
\item
We have $[L,L^2]=\{0\}$.
\end{enumerate}
Then the implications `(1)$\Rightarrow$(2)$\Rightarrow$(3)' hold. 
If $R$ is semiprime, then the implication `(3)$\Rightarrow$(2)' holds. 
Moreover, if $R$ is semiprime and has no 2-torsion, then all conditions are equivalent. 
\end{prp}
\begin{proof}
It is clear that~(1) implies~(2). 
If $[L,L]=0$, namely if every element of $L$ commutes with $L$, then clearly also every element of $L$ commutes with products of elements of $L$, so~(2) implies~(3). 
\vspace{.2cm}

Assume that $R$ is semiprime. We will show that (3) 
implies (2), using an argument
inspired by the proof of \cite[Lemma~1]{LanMon72LieStrPrimeChar2}.
We first establish the following:

Claim: \emph{For every  $x,y \in L$, we have $[x,y]^2 = 0$.}
Using that $y$ commutes with $yx$, we have $yxy=yyx$.
Similarly, we have $xxy = xyx$.
Using this at the last step, we get
\[
[x,y]^2 
= xyxy - xyyx - yxxy + yxyx
= x(yxy-yyx) - y(xxy - xyx)
= 0,
\]
which proves the claim.

Now, let $x,y \in L$ and set $u := [x,y]$.
We will show that $u = 0$.
Let $a \in R$.

Claim: \emph{We have $(ua)^4=0$}.
Applying the Jacobi identity, we have
\[\label{eqn:2.1}\tag{2.1}
[u,a]
= \big[[x,y],a\big]
= \big[x,[y,a]\big] + \big[y,[a,x]\big].
\]
Since $L$ is a Lie ideal, it follows that both $[y,a]$ and 
$[a,x]$ belong to $L$. Thus
$[x,[y,a]]$ and $[y,[a,x]]$ belong to $[L,L]$, and hence their squares are zero by the above claim.
Further, $[x,[y,a]]$ and $[y,[a,x]]$ both belong to $L$ (again because $L$ is a Lie ideal), and they also belong to $L^2$. 
In particular, they commute.
Set $s=[x,[y,a]]$ and $t=[y,[a,x]]$, and note that $s^2=t^2=0$ by
the claim above. 
Using that~$s$ and~$t$ commute at the second step, it follows that
\[
[u,a]^3
\stackrel{\eqref{eqn:2.1}}{=} \big( \big[x,[y,a]\big] + \big[y,[a,x]\big] \big)^3
= s^3+3s^2t+3st^2+t^3=0.
\]
Using that $u^2 = 0$ (by the claim) at the third step, one easily checks that
\[
0 = [u,a]^3ua
= (ua - au)^3 ua
= (ua)^4,
\]
as desired.

It follows that $\{u\} \cup uR$ is a right ideal of $R$ that is nil of bounded index at most $4$.
By Herstein's generalization of Levitzki's theorem, \cite[Lemma~1.1]{Her69TopicsRngThy}, it follows that $u=0$;
see also \cite[Exercise~10.13]{Lam01FirstCourse2ed}.
\vspace{.2cm}

Suppose that $R$ is semiprime and has no 2-torsion. 
To show that~(2) implies~(1), assume that $[L,L]=\{0\}$.
Let $x \in L$ and $a \in R$.
Since~$L$ is a Lie ideal, we have $[x,a] \in L$, and thus $[x,[x,a]]=0$.
Thus, $[x,[x,R]]=\{0\}$.
By \cite[Theorem~1]{Her70LieStructure}, if $M$ is a Lie ideal in a semiprime ring $S$ without $2$-torsion, and if $y \in S$ satisfies $[y,[y,M]]=\{0\}$, then $[y,M]=\{0\}$.
Applied to $y=x$ and $M=R$, it follows that $[x,R]=\{0\}$.
Thus, $[L,R]=\{0\}$, as desired.
\end{proof}

In \autoref{prp:CommutativitySemiprimeNo2Torsion}, condition~(2) need not imply~(1) in semiprime rings with $2$-torsion, as we show in the following example.

\begin{exa}
Let $\mathbb{F}$ be a field of characteristic $2$, set $R := M_2(\mathbb{F})$, the ring of $2$-by-$2$ matrices over $\mathbb{F}$, and 
\[
L := \left\{ 
\begin{pmatrix}
\mu & \lambda \\
\lambda & \mu \\
\end{pmatrix}\colon \mu, \lambda \in \mathbb{F}
\right\}.
\]
Then $R$ is semiprime and $[L,L]=\{0\}$ (so that $[L,L^2]=0$
as well).
Moreover, using that $\mathbb{F}$ has characteristic 2, a direct computation shows
that $L$ is a Lie ideal in $R$.
On the other hand, since the center of~$R$ consists only of scalar multiples of the identity, we see that $L$ is not contained in the center and thus $[L,R]\neq\{0\}$.
\end{exa}

We now specialize to the case that the Lie ideal is the commutator subgroup.

In preparation for the proof of the next result, let us recall that the characteristic of a ring $R$ is defined as the smallest natural number $n \in \NN$ with $n \geq 1$ such that $nR = \{0\}$, and as $0$ if no such $n$ exists.
Note that a prime ring~$R$ has characteristic $n \geq 1$ if and only if there exists $a\in R\setminus\{0\}$ with $na = 0$.
Indeed, given any $b \in R$, we have $aR(nb) = (na)Rb = \{0\}$.
Using that $R$ is prime, we have $a = 0$ or $nb = 0$, but $a \neq 0$ and so $nb = 0$.
In particular, a prime ring has characteristic $0$ if and only if it has no torsion.

\begin{thm}
\label{prp:CommutativitySemiprime}
Let $R$ be a semiprime ring.
Then the following are equivalent:
\begin{enumerate}
\item
The ring $R$ is commutative.
\item
We have $\big[[R,R],R]=\{0\}$, that is, $[R,R]$ is a subset of the center of $R$.
\item
We have $\big[[R,R],[R,R]^2]=\{0\}$.
\item
We have $\big[[R,R],[R,R]\big]=\{0\}$.
\end{enumerate}
\end{thm}
\begin{proof}
It is clear that~(1) implies~(2), and that~(2) implies~(3).
By \autoref{prp:CommutativitySemiprimeNo2Torsion}, (3) implies~(4).

\medskip

Let us show that~(4) implies~(1).
To verify that $R$ is commutative, let $x,y \in R$.
We will show that $[x,y] = 0$.
Since $R$ is semiprime, the intersection of all of its prime ideals is $\{0\}$. Thus, to show that $[x,y] = 0$, it suffices to show that $[x,y]$ belongs to every
prime ideal of $R$.
Let $I \subseteq R$ be a prime ideal.
Then the quotient ring $Q := R/I$ is prime and satisfies $\big[[Q,Q],[Q,Q]\big]=\{0\}$.
In order to show that $[x,y]=0$, we will show that 
$Q$ is commutative. In other words, and this is 
indeed what we will do, we may assume that $R$ is 
\emph{prime}.
\vspace{.2cm}

Let $Z$ denote the center of $R$.
Using that $[R,R]$ is a Lie ideal of $R$ satisfying $[[R,R],[R,R]]=\{0\}$, it follows from \cite[Theorem~4]{LanMon72LieStrPrimeChar2} that $[R,R] \subseteq Z$, unless we have the following exceptional case:
$R$ has characteristic $2$, $Z\neq\{0\}$, and the localization $RZ^{-1}$ of $R$ at $Z$ is isomorphic to the ring 
$M_2(\mathbb{F})$ of $2$-by-$2$ matrices over the quotient field $\mathbb{F}$ of $Z$.
We claim that the exceptional case cannot happen under our assumptions. Arguing by contradiction, assume that it does.
Then $\mathbb{F}$ is a field of characteristic $2$.
The assumption $\big[[R,R],[R,R]\big]=\{0\}$ implies that $\big[[RZ^{-1},RZ^{-1}],[RZ^{-1},RZ^{-1}]\big]=\{0\}$ as well.
Set $S := M_2(\mathbb{F})$ and 
\[
M := \left\{
\begin{pmatrix}
e & f \\
g & e \\
\end{pmatrix} :
e,f,g \in \mathbb{F} \right\}.
\]
Since 
\[\left[\begin{pmatrix}
a & b \\
c & d \\
\end{pmatrix},\begin{pmatrix}
a' & b' \\
c' & d' \\
\end{pmatrix}\right]= 
\begin{pmatrix}
bc'-cb' & * \\
* & cb'-bc' \\
\end{pmatrix}
\]
and $\mathbb{F}$ has characteristic $2$, it follows that every commutator in $S$ belongs to $M$.
Further, for $e,f,g \in \mathbb{F}$, we have
\[
\left[ 
\begin{pmatrix}
e & 0 \\
g & 0 \\
\end{pmatrix}, 
\begin{pmatrix}
0 & 1 \\
0 & 0 \\
\end{pmatrix}
\right]
= \begin{pmatrix}
g & e \\
0 & g \\
\end{pmatrix}, \andSep
\left[ 
\begin{pmatrix}
e & f \\
0 & 0 \\
\end{pmatrix}, 
\begin{pmatrix}
0 & 0 \\
1 & 0 \\
\end{pmatrix}
\right]
= \begin{pmatrix}
f & 0 \\
e & f \\
\end{pmatrix}.
\]
Thus, $M$ is the commutator subgroup of $S$.
In particular, the matrices $\begin{psmallmatrix}
0 & 1 \\
0 & 0 \\
\end{psmallmatrix}$
and $\begin{psmallmatrix}
0 & 0 \\
1 & 0 \\
\end{psmallmatrix}$
belong to $M$, but their commutator does not vanish.
Thus, $[M,M]\neq\{0\}$.
This shows that $R$ is not of the exceptional case.

Thus, we have $[R,R] \subseteq Z$.
We conclude that $R\subseteq Z$ by \cite[Lemma~7]{LanMon72LieStrPrimeChar2}, and hence~$R$ is commutative.
\end{proof}

\begin{cor}
\label{prp:CommutatorIdealNil}
Let $R$ be a ring such that $\big[[R,R],[R,R]^2\big]=\{0\}$.
Then the commutator ideal of $R$ is nil.
\end{cor}
\begin{proof}
Let $P$ denote the lower nilradical (or prime radical) of $R$.
Then $P$ is a nil ideal and $R/P$ is semiprime;
see \cite[Definition~10.13]{Lam01FirstCourse2ed}.
It follows from \autoref{prp:CommutativitySemiprime} that $R/P$ is commutative.
Hence, the commutator ideal of $R$ is contained in $P$ and therefore nil.
\end{proof}

\section{Rings generated by commutators}

In this section, we show that a unital ring is generated by its commutators as an ideal if and only if every element can be written as a sum of products of pairs of commutators.
Moreover, we show that there is a universal bound on the number of summands required;
see \autoref{prp:UnitalAlgGenCommut}.
In the general case of a not necessarily unital ring, we show that every element in the commutator ideal has a power that is a sum of products of pairs of commutators;
see \autoref{prp:NonunitalAlgGenCommut}.

\medskip

We actually work in the more general setting of algebras over 
rings, so we fix a unital commutative ring $K$ and will consider
$K$-algebras. The case of rings is obtained by specializing to 
$K=\ZZ$.

Given a subset $X\subseteq A$ in a $K$-algebra $A$ 
and $n \in \NN$, we use
\[
\sum{}^n X := \big\{ x_1 + x_2 + \ldots + x_n\in A\colon x_1,\ldots,x_n \in X \big\}
\]
to denote the collection of elements in $A$ that can be written as a sum of $n$ elements from $X$.

A \emph{Lie ideal} in a $K$-algebra $A$ is a $K$-linear subspace $L \subseteq A$ such that $[A,L] \subseteq L$.
Given Lie ideals $L,M \subseteq A$, it is easy to show, using the Jacobi identity, that the subspace $[L,M]$ is a Lie ideal as well.

Given subsets $X,Y \subseteq A$, recall from the beginning of \autoref{sec:Commutativity} that 
\[[X,Y]_1 := \{[x,y]\colon x \in X, y \in Y\} \ \ \mbox{ and } \ \ X \multNoSpan Y := \{xy\colon x \in X, y \in Y\}.\]
We further set $X^{\multNoSpan 2} := X \multNoSpan X$.

\begin{lma}
\label{prp:Ideal-LeftIdeal}
Let $L,M$ be Lie ideals in a $K$-algebra $A$.
Then
\[
A[L,M]A 
\ \subseteq \ A[L,M].
\]
Moreover, if $L_0 \subseteq L$ and $M_0 \subseteq M$ are subsets, 
and $m,n \in \NN$ satisfy 
\[
[A,L_0]_1 \subseteq \sum{}^n L_0, \andSep 
[A,M_0]_1 \subseteq \sum{}^m M_0,
\] 
then $A \multNoSpan [L_0,M_0]_1 \multNoSpan A \subseteq \sum^{1+n+m} A \multNoSpan [L_0,M_0]_1$.
\end{lma}
\begin{proof}
We only prove the quantitative statement since the proof of the non-quan\-ti\-ta\-tive statement is analogous.
Let $a,b \in A$, let $x \in L_0$, and let $y \in M_0$.
It follows from the Jacobi identity that
\[
\big[[x,y],b\big]
= - \big[[y,b],x\big] - \big[[b,x],y\big]
= \big[[b,y],x\big] + \big[[x,b],y\big].
\]
and therefore
\begin{align*}
a[x,y]b
&= a\big[[x,y],b\big] + ab[x,y] \\
&= a\big[[b,y],x\big] + a\big[[x,b],y\big] + ab[x,y]
\in \sum{}^{1+n+m} A \multNoSpan [L_0,M_0]_1,
\end{align*}
as desired.
\end{proof}

It is a folklore result from Lie theory that if $L$ is a Lie ideal in an algebra, then the ideal generated by $[L,L]$ is contained in $L+L^2$;
see, for example \cite[Lemma~1.1]{Rob16LieIdeals}.
We include the result together with a quantitative statement for future reference.

\begin{lma}
\label{prp:Ideal-L-L}
Let $L$ be a Lie ideal in a $K$-algebra $A$.
Then
\[
A[L,L]A 
\ \subseteq \ A[L,L] 
\ \subseteq \ L + L^2.
\]
Moreover, if $L_0 \subseteq L$ and $n \in \NN$ satisfy $[A,L_0]_1,[L_0,A]_1 \subseteq \sum^n L_0$, then 
\[
A \multNoSpan [L_0,L_0]_1 
\subseteq \sum{}^{n} (L_0 + L_0^{\multNoSpan 2}), \andSep
A \multNoSpan [L_0,L_0]_1 \multNoSpan A 
\subseteq \sum{}^{n(1+2n)} (L_0 + L_0^{\multNoSpan 2}).
\]
\end{lma}
\begin{proof}
We only prove the quantitative statement since the proof of the non-quan\-ti\-ta\-tive statement is analogous.
Given $a \in A$ and $x,y \in L_0$, we have
\[
a[x,y]
= [ax,y] + [y,a]x
\in [A,L_0]_1 + [L_0,A]_1 \multNoSpan L_0
\subseteq \sum{}^{n} (L_0+L_0^{\multNoSpan 2}).
\]

Applying \autoref{prp:Ideal-LeftIdeal} at the first step with $M_0 = L_0$ and $m=n$, we obtain
\[
A \multNoSpan [L_0,L_0]_1 \multNoSpan A 
\subseteq \sum{}^{1+2n} A \multNoSpan [L_0,L_0]_1
\subseteq \sum{}^{n(1+2n)} (L_0 + L_0^{\multNoSpan 2}),
\]
as desired.
\end{proof}

The next result is a variation of \autoref{prp:Ideal-L-L} that allows us to strengthen containment in $L+L^2$ to containment in $L^2$.

\begin{lma}
\label{prp:Ideal-L-L2}
Let $L$ be a Lie ideal in a $K$-algebra $A$ such that $[A,L^2] \subseteq L$.
(For example, this is automatically the case if $L$ contains $[A,A]$.)
Then:
\[
A[L,L^2]A 
\ \subseteq \ A[L,L^2] 
\ \subseteq \ L^2.
\]
Moreover, if $L_0 \subseteq L$ and $m,n \in \NN$ satisfy $[A,L_0]_1 \subseteq \sum^n L_0$ and $[L_0^{\multNoSpan 2},A]_1 \subseteq \sum^m L_0$, then
\[
A \multNoSpan [L_0,L_0^{\multNoSpan 2}]_1 
\subseteq \sum{}^{(2n+m)} L_0^{\multNoSpan 2}, \andSep
A \multNoSpan [L_0,L_0^{\multNoSpan 2}]_1 \multNoSpan A 
\subseteq \sum{}^{(2n+m)(1+3n)} L_0^{\multNoSpan 2}.
\]
\end{lma}
\begin{proof}
Again we only prove the quantitative statement.
For $a \in A$ and $x,y \in L_0$, we have
\[
[a,xy] 
= [a,x]y + x[a,y] 
\in \Big(\sum{}^nL_0\Big) \multNoSpan L_0 + L_0 \multNoSpan \Big(\sum{}^nL_0\Big)
= \sum{}^{2n} L_0^{\multNoSpan 2}.
\]
Thus, setting $M_0 = L_0^{\multNoSpan 2}$, we get $[A,M_0]_1 \subseteq \sum^{2n} M_0$. 
Since $[A,L_0]_1 \subseteq \sum{}^{n} L_0$, it follows from \autoref{prp:Ideal-LeftIdeal} that 
\[
\tag{3.1} 
A \multNoSpan [L_0,L_0^{\multNoSpan 2}]_1 \multNoSpan A 
\subseteq \sum{}^{(1+3n)} A \multNoSpan [L_0,L_0^{\multNoSpan 2}]_1.
\]

Given $a \in A$ and $x,y,z \in L_0$, using that $[A,L_0^{\multNoSpan 2}]_1 \subseteq \sum^{2n} L_0^{\multNoSpan 2}$ and $[A,L_0^{\multNoSpan 2}]_1 \multNoSpan L_0 \subseteq \sum^m L_0^{\multNoSpan 2}$, we have
\[
a[x,yz]
= [ax,yz] + [yz,a]x
\in [A,L_0^{\multNoSpan 2}]_1 + [L_0^{\multNoSpan 2},A]_1 \multNoSpan L_0
\subseteq \sum{}^{(2n+m)} L_0^{\multNoSpan 2}.
\]
Thus $A \multNoSpan [L_0,L_0^{\multNoSpan 2}]_1 \subseteq \sum{}^{(2n+m)} L_0^{\multNoSpan 2}$. 
Using this at the second step, we get
\[
A \multNoSpan [L_0,L_0^{\multNoSpan 2}]_1 \multNoSpan A 
\stackrel{\mathrm{(3.1)}}{\subseteq} \sum{}^{(1+3n)} A \multNoSpan [L_0,L_0^{\multNoSpan 2}]_1
\subseteq \sum{}^{(2n+m)(1+3n)} L_0^{\multNoSpan 2},
\]
as desired.
\end{proof}

\begin{thm}
\label{prp:UnitalAlgGenCommut}
Let $A$ be a unital $K$-algebra.
Then the following are equivalent:
\begin{enumerate}
\item
We have $A = A[A,A]A$, that is, $A$ agrees with its commutator ideal.
\item
The algebra~$A$ is generated by its commutators as a $K$-algebra.
\item
There exists $N \in \NN$ such that for every $a \in A$ there exist $b_j,c_j,d_j,e_j \in A$ for $j = 1,\ldots,N$ such that
\[
a = \sum_{j=1}^N [b_j,c_j][d_j,e_j].
\]
\end{enumerate}
\end{thm}
\begin{proof}
It is clear that~(3) implies~(2), which in turn implies~(1).
To show that~(1) implies~(3), assume that $A = A[A,A]A$.
Set 
\[
L := [A,A], \andSep
I := A[L,L^2]A.
\] 

We first show that $I=A$.
To reach a contradiction, assume that $I \neq A$.
Using that $A$ is unital, choose a maximal ideal $J \subseteq A$ with $I \subseteq J$.
Set $B := A/J$, which is a simple and unital $K$-algebra, and therefore a prime ring.

We claim that $\big[[B,B],[B,B]^2\big]=\{0\}$. 
To see this, let $b_1,\ldots,b_6\in B$. 
We want to show that $\big[[b_1,b_2],[b_3,b_4][b_5,b_6]\big]=0$.
Let $\pi\colon A\to B$ denote the quotient map, and find $a_1,\ldots,a_6\in A$ such that $\pi(a_j)=b_j$ for all $j=1,\ldots,6$. 
Then $\big[[a_1,a_2],[a_3,a_4][a_5,a_6]\big]$ belongs to $I$ by definition (and since $A$ is unital), and thus 
\[
\big[[b_1,b_2],[b_3,b_4][b_5,b_6]\big]
= \pi\big(\big[[a_1,a_2],[a_3,a_4][a_5,a_6]\big]\big)=0,
\]
as desired.

Thus, the prime ring $B$ satisfies $\big[[B,B],[B,B]^2]=\{0\}$.
Applying \autoref{prp:CommutativitySemiprime}, we deduce that $B$ is commutative.
Since, by assumption, the algebra~$A$ is generated by its commutators as an ideal, the same is true for $B$. 
This is a contradiction, which shows that $I = A$.

Set $L_0 := [A,A]_1$.
It follows that $A$ is the linear span of $A \multNoSpan [L_0,L_0^2]_1 \multNoSpan A$, and we obtain $d \in \NN$ and $a_j,b_j \in A$ and $x_j, y_j, z_j\in L_0$, for $j=1,\ldots,d$, such that
\[
1 = \sum_{j=1}^d a_j[x_j,y_jz_j]b_j.
\]

We claim that every element in $A$ is a sum of at most $12d$ products of pairs of commutators.
To show this, note that $[A,L_0]_1 \subseteq L_0$ and $[A,L_0^2]_1 \subseteq L_0$. 
Applying \autoref{prp:Ideal-L-L2} with $m=n=1$ gives 
\[
\tag{3.2}
A \multNoSpan [L_0,L_0^2]_1 \multNoSpan A 
\subseteq \sum{}^{12} L_0^{\multNoSpan 2}.
\]
Let $a\in A$.
Then
\[
a = \sum_{j=1}^d aa_j[x_j,y_jz_j]b_j
\in \sum{}^d A \multNoSpan [L_0,L_0^{\multNoSpan 2}]_1 \multNoSpan A
\stackrel{\mathrm{(3.2)}}{\subseteq} \sum{}^{12d} L_0^{\multNoSpan 2}.
\]
This verifies~(3) for $N := 12d$, and finishes the proof.
\end{proof}

In Sections~\ref{sec:BoundsRing} and~\ref{sec:BoundsCAlg}, we will study the minimal $N$ for which~(3) in \autoref{prp:UnitalAlgGenCommut} is satisfied.

\medskip

We now turn to not necessarily unital rings and $K$-algebras.
In this case, we show that every element in the commutator ideal has a power that is a sum of products of pairs of commutators.

Given a non-unital $K$-algebra $A$, we let $\widetilde{A}$ denote the unitization, which is given by $\widetilde{A} = K \times A$ with coordinatewise addition and scalar multiplication, and multiplication defined as $(\lambda,x)(\mu,y) = (\lambda\mu,\mu x + \lambda y + xy)$ for $\lambda,\mu \in K$ and $x,y \in A$.
Then $\widetilde{A}$ is a unital $K$-algebra, with unit $(1,0)$, and the canonical map $A \to \widetilde{A}$ identifies $A$ with an ideal in $\widetilde{A}$.
If $A$ is unital, we set $\widetilde{A} := A$.
See \cite[Section~2.3]{Bre14BookIntroNCAlg} for details.
The minimal unitization of a \ca{} $A$ agrees with this construction when $A$ is viewed as a $\CC$-algebra.

Given a  $K$-algebra $A$, the ideal generated by a subspace $L$ is $L+AL+LA+ALA$, which agrees with $\widetilde{A}L\widetilde{A}$ viewed inside $\widetilde{A}$.
Note that $ALA$ is also an ideal in $A$, which, however, may not contain $L$ if $A$ is not unital.

\begin{thm}
\label{prp:NonunitalAlgGenCommut}
Let $A$ be a $K$-algebra and consider the ideals
\[
I := \widetilde{A}\big[[A,A],[A,A]^2\big]\widetilde{A} \andSep 
J := \widetilde{A}[A,A]\widetilde{A}.
\]

Then $I \subseteq [A,A]^2 \subseteq J$, and $J/I$ is nil.
Thus, for every $a,d \in \widetilde{A}$ and $b,c \in A$, there exist $n,m \geq 1$ and $w_j,x_j,y_j,z_j \in A$ for $j = 1,\ldots,n$ such that
\[
\big( a[b,c]d \big)^m = \sum_{j=1}^n [w_j,x_j][y_j,z_j].
\]

In particular, if $A = \widetilde{A}[A,A]\widetilde{A}$, then for every $a \in A$, there exist $n,m \geq 1$ and $b_j,c_j,d_j,e_j \in A$ for $j = 1,\ldots,n$, such that
\[
a^m = \sum_{j=1}^n [b_j,c_j][d_j,e_j].
\]
\end{thm}
\begin{proof}
That $I$ is contained in $[A,A]^2$ follows from \autoref{prp:Ideal-L-L2}, while the inclusion $[A,A]^2 \subseteq J$ is clear.
Note that the quotient $B := J/I$ satisfies $\big[[B,B],[B,B]^2\big]=\{0\}$ and is generated by its commutators as an ideal.
It thus follows from \autoref{prp:CommutatorIdealNil} that $B$ is nil, as desired.
\end{proof}

\section{C*-algebras generated by commutators}
\label{sec:CAlgGenCommutator}

In this section, we prove that if a \ca{} is generated by its commutators as a (not necessarily closed) ideal, then every element is a sum of products of pairs of commutators;
see \autoref{prp:CommutatorsCAlg}.

\medskip

We use $\widetilde{A}$ to denote the minimal unitization of a \ca{} $A$;
see the comments before \autoref{prp:NonunitalAlgGenCommut}.
The (not necessarily closed) ideal generated by a subspace $L \subseteq A$ is $\widetilde{A}L\widetilde{A}$.
While in general it is not clear if $ALA = \widetilde{A}L\widetilde{A}$ since~$L$ may not be a subset of $ALA$, for the special case $L=[A,A]$ we showed in \cite[Proposition~3.2]{GarThi23arX:PrimeIdealsCAlg} that $[A,A] \subseteq A[A,A]A$ and therefore $A[A,A]A = \widetilde{A}[A,A]\widetilde{A}$.

A \emph{character} on a \ca{} $A$ is a one-dimensional, irreducible representation. (Equivalently, a nonzero homomorphism $A\to \mathbb{C}$.)

\begin{thm}
\label{prp:CommutatorsCAlg}
Let $A$ be a \ca.
Then the following are equivalent:
\begin{enumerate}
\item
We have $A=\widetilde{A}[A,A]\widetilde{A}$, that is, $A$ is generated by its commutators as a (not necessarily closed) ideal.
\item
We have $A=A[A,A]A$, that is, for every $a \in A$ there exist $n \in \NN$ and $b_j,c_j,d_j,e_j \in A$, for $j=1,\ldots,n$, such that
\[
a = \sum_{j=1}^n b_j[c_j,d_j]e_j.
\]
\item
We have $A=A\big[[A.A],[A.A]\big]A$.
\item
The \ca{} $A$ is generated by its commutators as a ring.
\item
We have $A=[A,A]^2$, that is, for every $a \in A$ there exist $n \in \NN$ and  $b_j,c_j,d_j,e_j \in A$, for $j=1,\ldots,n$, such that
\[
a = \sum_{j=1}^n [b_j,c_j][d_j,e_j].
\]
\end{enumerate}
Moreover, these conditions imply:
\begin{enumerate}
\setcounter{enumi}{5}
\item
The \ca{} $A$ admits no character.
\end{enumerate}
Moreover, (6) implies that $A[A,A]A$ is dense in $A$,
and thus all conditions are equivalent if $A$ is unital.
\end{thm}
\begin{proof}
By \cite[Theorem~3.3]{GarThi23arX:PrimeIdealsCAlg}, (1), (2) and~(3) are equivalent.
It is clear that~(5) implies~(4), which in turn implies~(1).

Let us show that~(1) implies~(5).
Thus, assuming that $A = \widetilde{A}[A,A]\widetilde{A}$, consider the ideal
\[
I := \widetilde{A}\big[[A,A],[A,A]^2\big]\widetilde{A}.
\]
Applying \autoref{prp:NonunitalAlgGenCommut}, we obtain that
\[
I \subseteq [A,A]^2 \subseteq A,
\]
and $A/I$ is nilradical.
Thus~$A$ is a radical extension over $I$.
Since \ca{s} are never radical extension over proper ideals (\cite[Proposition~2.2]{GarThi23arX:PrimeIdealsCAlg}), we get $I = A$, and hence $A = [A,A]^2$, as desired.

We have seen that~(1)-(5) are equivalent.
It is easy to see that every character vanishes on commutators, which shows that~(4) implies~(6).

We will now show that (6) implies $\overline{A[A,A]A}=A$. 
To reach a contradiction, assume that $J := \overline{A[A,A]A}$ is a proper subset of~$A$, and 
thus a proper (closed, two-sided) ideal. 
The quotient $A/J$ is a commutative \ca{} and therefore admits a character.
Composed with the quotient map $A \to A/J$, we obtain a character for~$A$, which is the desired contradiction.

The last claim follows from the fact that the only dense ideal in a unital \ca{} is the \ca{} itself, and hence (6) implies (2) in this case.
\end{proof}

\begin{rmks}
\label{rmk:ReplaceCommutatorsByTheirProducts}
(1)
It was shown in \cite[Lemma~2.5]{BreFosFos05JordanIdlsRevisited} that a unital \ca{} is generated by its commutators as an ideal if and only if it is generated by its commutators as an algebra, and if and only if it admits no characters.
\autoref{prp:CommutatorsCAlg} refines and generalizes this result.

(2)
Let $A$ be a unital \ca{} that admits no characters.
Robert showed in \cite[Theorem~3.2]{Rob16LieIdeals} that every element in $A$ can be written as
\[
\sum_{j=1}^N [a_j,b_j] + \sum_{j=1}^N [c_j,d_j][e_j,f_j]
\]
for some $N \in \NN$ and elements $a_j,b_j,c_j,d_j,e_j,f_j \in A$,
for $j=1,\ldots,N$.
\autoref{prp:CommutatorsCAlg} shows that the summands $[a_j,b_j]$ are not necessary, and every element can be written as a sum of products of commutators.

It is not known if every every commutator in a \ca{} is a linear combination of products of pairs of commutators;
see \autoref{qst:Commutators}(c).
\end{rmks}

\begin{qst}
\label{qst:BoundNonunitalCAlg}
Let $A$ be a nonunital \ca{} that is generated by its commutators as an ideal.
Is there a uniform upper bound on the number of summands necessary in statement~(3) in \autoref{prp:CommutatorsCAlg}?
\end{qst}

The following example was provided by Ozawa, who kindly allowed us to include it here.
It shows that~(6) does not imply (1)-(5) in \autoref{prp:CommutatorsCAlg} for nonunital \ca{s}. 
In the example, we use the \ca{ic} product of \ca{s} $A_n$ for $n\in\NN$, which is defined as 
\[
\prod_{n\in\NN}A_n=\Big\{(a_n)_{n\in\NN}\colon a_n\in A_n, \sup_{n\in\NN}\|a_n\|<\infty\Big\}.
\]

\begin{exa}
\label{exa:NoCharNotGenCommutators}
By \cite[Corollary~8.6]{RobRor13Divisibility}, there exists a sequence $(A_n)_{n\in\NN}$ of unital, simple, infinite-dimensional \ca{s} $A_n$ such that the product $\prod_{n\in\NN} A_n$ has a character.
We claim that for every $m \geq 1$ and $C > 1$, there exists $n = n(m,C)$ such that in $A_n$ the following holds:
If $1 = \sum_{j=1}^m a_j[b_j,c_j]d_j$ for $a_j,b_j,c_j,d_j \in A_n$, then $\sum_{j=1}^m \|a_j\|\|b_j\|\|c_j\|\|d_j\| > C$.

Indeed, if this is not the case, then there exist $m \geq 1$ and $C > 1$ such that for every $n \in \NN$ there are $a_{n,j},b_{n,j},c_{n,j},d_{n,j} \in A_n$ such that
\[
1= \sum_{j=1}^m a_{n,j}[b_{n,j},c_{n,j}]d_{n,j}, \andSep
\sum_{j=1}^m \|a_{n,j}\|\|b_{n,j}\|\|c_{n,j}\|\|d_{n,j}\| \leq C.
\]
One may arrange that $\| a_{n,j} \| = \| b_{n,j} \| = \| c_{n,j} \| = \| d_{n,j} \|$ for each $n$ and $j$.
It follows that 
\[
a_j := (a_{n,j})_n, \quad
b_j := (b_{n,j})_n, \quad
c_j := (c_{n,j})_n, \andSep
d_j := (d_{n,j})_n
\]
are bounded sequences and thus belong to $\prod_{n\in\NN} A_n$.
Then $1 = \sum_{j=1}^m a_j[b_j,c_j]d_j$ in $\prod_{n\in\NN} A_n$, which contradicts that $\prod_{n\in\NN} A_n$ has a  one-dimensional irreducible representation.

For $m\in\NN$, consider the natural number $n(m,m^3)$ as in the claim for $m=m$ and $C=m^3$.
Let us consider the direct sum $B := \bigoplus_m A_{n(m,m^3)}$.
Since each summand $A_{n(m,m^3)}$ has no character, neither does $B$.
On the other hand, we claim that~$B$ does not agree with its commutator ideal.
To see this, assume that the element $e := (\tfrac{1}{m})_{m\in\NN}$ belongs to the commutator ideal.
Then there exist $r\in\NN$ and sequences 
$a_k=(a_{m,k})_{m\in\NN}, b_k=(b_{m,k})_{m\in\NN}, c_k=(c_{m,k})_{m\in\NN}, d_k=(d_{m,k})_{m\in\NN} \in B$ with $e = \sum_{k=1}^r a_k[b_k,c_k]d_k$.
Choose $m$ such that $m \geq r$ and $m^2 \geq \sum_{k=1}^r \|a_k\| \|b_k\| \|c_k\| \|d_k\|$.
In the quotient $A_{n(m,m^3)}$ of $B$, we have
\[
\tfrac{1}{m} = \sum_{k=1}^r a_{m,k}[b_{m,k},c_{m,k}]d_{m,k}, \andSep
\sum_{k=1}^r \|a_{m,k}\| \|b_{m,k}\| \|c_{m,k}\| \|d_{m,k}\| \leq m^2,
\]
a contradiction to the choice of $n(m,m^3)$.
\end{exa}

Given a \ca{} $A$, we use $A_0$ to denote the set of self-adjoint elements in $A$ that vanish under every tracial state on $A$.
A \emph{self-commutator} in $A$ is an element of the form $[x^*,x]$ for some $x \in A$.
Clearly, every self-commutator belongs to $A_0$, and consequently every norm-convergent sum $\sum_{j=1}^\infty [x_j^*,x_j]$ belongs to $A_0$.
It was shown by Cuntz and Pedersen, \cite{CunPed79EquivTraces}, that all elements of $A_0$ arise this way.

Thus, every element in $A_0$ can be approximated in norm by finite sums of self-commutators.
However, the following example shows that there is no uniform bound on the number of summands needed -- even if $A$ is unital.
This should be contrasted with \autoref{prp:UnitalAlgGenCommut}.

\begin{exa}
Based on examples by Pedersen and Petersen \cite{PedPet70IdealsCAlg}, Bice and Farah showed in \cite[Theorem~2.1]{BicFar15Trace} that for each $m$ there exist a unital \ca{} $B$ and an element $b \in B_0$ that is not a limit of sums of $m$ self-commutators. 

Robert showed that this phenomenon can even be accomplished in a simple \ca{} for all $m$ simultaneously:
By \cite[Theorem~1.4]{Rob15NucDimSumsCommutators}, there exists a simple, separable, unital \ca{}~$A$ with a unique tracial state $\tau$ that contains contractive elements $a_m \in A_0$, for $m \geq 1$, satisfying $\big\| a_m - \sum_{j=1}^m [x_j^*,x_j]\big \| \geq 1$ for all $x_1,\ldots,x_m \in A$.

Let us see that $\dist(a_{4m}, \sum^{m}[A,A]_1)=1$.
Since $a_{4m}$ is contractive, it suffices to show that $\dist\big(a_{4m}, \sum^{m}[A,A]_1\big) \geq 1$.
To reach a contradiction, assume that there is $c \in \sum^m  [A,A]_1$ with $\| a_{4m} - c \| < 1$.
Then $\big\| a_{4m} - \tfrac{1}{2}(c+c^*) \big\| < 1$.
Note that $\tfrac{1}{2}(c+c^*)$ is a self-adjoint element in $\sum^{2m}[A,A]_1$.
It follows from \cite[Theorem~2.4]{Mar06SmallNrCommutators} (and its proof), that every self-adjoint element in $\sum^k [A,A]_1$ is a sum of at most $2k$ self-commutators, for every $k \in \NN$.
Thus, $a_{4m}$ has distance strictly less than $1$ to a sum of at most $4m$ self-commutators, which is the desired contradiction.

It follows that $A_0$ is not contained in the closure of $\sum^m [A,A]_1$ for any $m \in \NN$.
On the other hand, since $A$ is unital, by \autoref{prp:UnitalAlgGenCommut} there exists $N\in\NN$ such that $A = \sum^N [A,A]_1^{\multNoSpan 2}$.
\end{exa}

\section{Bounds on the number of commutators for rings}
\label{sec:BoundsRing}

In \autoref{prp:UnitalAlgGenCommut} we have seen that if a unital ring $R$ is generated by its (additive) commutators, then there exists $N \in \NN$ such that $R = \sum^N [R,R]_1^{\multNoSpan 2}$, where $[R,R]_1$ denotes the set of commutators in $R$, and $[R,R]_1^{\multNoSpan 2}$ denotes the set of products of pairs of commutators.
We define the invariant $\xi(R)$ as the minimal $N$ with this property;
see \autoref{dfn:Xi}.
We show that $\xi(M_n(S)) \leq 2$ for the ring $M_n(S)$ of $n$-by-$n$ matrices over an arbitrary untial ring $S$ and $n \geq 2$;
see \autoref{prp:MatrixSizeN}.
For matrix algebras over a field $\mathbb{F}$, we prove that $\xi(M_n(\mathbb{F})) = 1$, which amounts to showing that every matrix is a product of two matrices of trace zero;
see \autoref{exa:MatrixAlgebra}.

We then show that every noncommutative division ring $D$ satisfies $\xi(D) \leq 2$ (\autoref{exa:DivisionRing}), and deduce that $\xi(R) \leq 2$ for every semisimple, Artinian ring $R$ that has no commutative summand (\autoref{exa:Semisimple}). 

In \autoref{prp:EstimateFine}, we develop a method to obtain good bounds for $\xi(R)$ for a unital ring $R$ from the existence of special elements in a unital subring of $R$.
Using this, we prove that $\xi(R) \leq 3$ as soon as $R$ admits a unital ring homomorphism from a direct sum of matrix rings;
see \autoref{prp:XiMatrixSubalgebra}.

\begin{dfn}
\label{dfn:Xi}
Given a unital ring $R$ generated by its commutators, we set
\[
\xi(R) := \min \left\{ N \in \NN : R = \sum{}^N [R,R]_1^{\multNoSpan 2} \right\}.
\]
\end{dfn}

\begin{rmks}
(1)
In \cite{Pop02FiniteSumsCommutators}, Pop shows that a unital \ca{} $A$ admits no tracial states if and only if there exists $n$ such that $A = \sum^n [A,A]_1$, and in this case he defines the invariant $\nu(A)$ as the minimal number $n$ such that $A = \sum^n [A,A]_1$.
Our invariant $\xi$ is defined analogously.

(2)
For a $K$-algebra $A$ over a unital, commutative ring $K$, it would arguably be more natural to define $\xi(A)$ as the smallest $N$ such that every element of $A$ is a \emph{$K$-linear} combination of (at most) $N$ elements in $[A,A]_1^{\multNoSpan 2}$.  
Since the set $[A,A]_1^{\multNoSpan 2}$ is $K$-invariant, this
alternative definition agrees with the one given above.

(3)
In \cite[Example~3.11]{Rob16LieIdeals}, Robert shows that for every $N \in \NN$ there exists a unital \ca{} $A$ such that not every element of $A$ can be expressed as a sum of~$N$ elements in $[A,A]_1^{\multNoSpan 2}$, that is, such that $\xi(A) > N$.
\end{rmks}

\begin{qst}
What are the possible values $\xi(R)$ for unital rings $R$?
What are the possible values for unital, simple rings?
\end{qst}

The rest of this work is concerned with finding upper bounds for $\xi(R)$ for different classes of rings $R$.
We begin with matrix rings.
In \cite[Theorem~15]{Mes06CommutatorRings}, Mesyan showed that for every unital ring $S$ and $n \geq 2$, every matrix in $M_n(S)$ with trace zero is a sum of two commutators.
Based on his ideas, we show that every matrix in $M_n(S)$,
regardless of its trace, is a sum of two products of pairs of commutators.

\begin{thm}
\label{prp:MatrixSizeN}
Let $S$ be a unital ring and let $n \geq 2$.
Then $\xi(M_n(S)) \leq 2$.
\end{thm}
\begin{proof}
Assume first that $n=2$. 
Given $x,y \in S$, we have
\[
\left[
\begin{pmatrix}
1 & 0 \\
0 & 0
\end{pmatrix},
\begin{pmatrix}
0 & x \\
-y & 0
\end{pmatrix}
\right]
=
\begin{pmatrix}
0 & x \\
y & 0
\end{pmatrix}
\in [M_2(S),M_2(S)]_1.
\]

Further, we have
\begin{align}
\label{eq:Diagonal1Minus1}\tag{5.1}
\left[
\begin{pmatrix}
0 & 1 \\
0 & 0
\end{pmatrix},
\begin{pmatrix}
0 & 0 \\
1 & 0
\end{pmatrix}
\right]
=
\begin{pmatrix}
1 & 0 \\
0 & -1
\end{pmatrix}
\in [M_2(S),M_2(S)]_1.
\end{align}

Therefore, given any $a,b,c,d \in S$, we have
\[
\begin{pmatrix}
a & b \\
c & d
\end{pmatrix}
=
\begin{pmatrix}
1 & 0 \\
0 & -1
\end{pmatrix}
\begin{pmatrix}
0 & b \\
-c & 0
\end{pmatrix}
+
\begin{pmatrix}
0 & 1 \\
1 & 0
\end{pmatrix}
\begin{pmatrix}
0 & d \\
a & 0
\end{pmatrix}
\in \sum{}^2 [M_2(S),M_2(S)]_1^{\multNoSpan 2}.
\]
\vspace{.2cm}

Assume now that $n \geq 3$.
For $j,k=1,\ldots,n$, let let $e_{j,k} \in M_n(\{0,1\})\subseteq M_n(S)$ denote the standard matrix unit.
Let $a \in M_n(S)$ and set 
\[
x := \sum_{j=1}^{n-1} e_{j+1,j}, \quad
y := \sum_{j=1}^{n-1} e_{j,j+1}, \andSep
c := a + xay + \ldots + x^{n-1}ay^{n-1}.
\]
As shown in the proof of \cite[Theorem~15]{Mes06CommutatorRings}, one can readily check that
\[
[cy,x] = a - ce_{n,n}.
\]
Note that the entries of the matrix $a - ce_{n,n}$ agree with those of $a$ in the first $n-1$ columns.

Using block-diagonal matrices with blocks 
$\begin{psmallmatrix}
0 & 1 \\
0 & 0
\end{psmallmatrix}$
and
$\begin{psmallmatrix}
0 & 0 \\
1 & 0
\end{psmallmatrix}$ 
and using the identity in \eqref{eq:Diagonal1Minus1}, we see that the diagonal matrix $d$ with entries $1,-1,\ldots,1,-1,0$ (for $n$ odd) and with entries $1,-1,\ldots,1,-1,0,0$ (for $n$ even) is a commutator in $M_n(S)$.
Moreover, we have $[cy,x]d = ad$, so $ad$ is a product of two 
commutators. Set $p_n=1$ if $n$ is odd and $p_n=2$ if $n$ is even.
Note that $ad$ is an arbitrary matrix in $M_n(S)$ whose last $p_n$-many columns are zero.

We have shown that every matrix whose last $p_n$-many columns are zero is a product of two commutators in $M_n(S)$.
Analogously, every matrix whose first $p_n$-many columns are zero is also a product of two commutators.
Since $n \geq 3$, every matrix is a sum of two such matrices, thus 
concluding the proof.
\end{proof}

For matrix algebras over a field $\mathbb{F}$, we have the optimal estimate $\xi(M_n(\mathbb{F})) = 1$.

\begin{exa}
\label{exa:MatrixAlgebra}
Let $\mathbb{F}$ be a field, and let $n \geq 2$.
For characteristic $0$, it was shown by Wu in \cite[Theorem~5.12]{Wu89OpFactorization} that every matrix in $M_n(\mathbb{F})$ is a product of two (additive) commutators.
Botha showed in \cite[Theorem~4.1]{Bot97ProdMatPrescribedTraces} that the same holds in positive characteristic.

Thus, we have $\xi(M_n(\mathbb{F})) = 1$. 
This means that the image of the polynomial $f(a,b,c,d)=[a,b][c,d]$ in $M_n(\mathbb{F})$ is all of $M_n(\mathbb{F})$.
In particular, this verifies the {L}'vov-Kaplansky conjecture for this polynomial, for all fields, and for all $n \geq 2$.
We refer the reader to the survey \cite{KanMalRowYav20EvalNcPolynomials} for background on the {L}'vov-Kaplansky conjecture.
\end{exa}

\begin{exa}
Let $R$ be a unital ring that admits a surjective inner derivation, that is, we have $R = \{ [a,x] \colon x \in R \}$ for some element $a \in R$.
This clearly implies that every element in $R$ is a product of two commutators and thus $\xi(R)=1$.
Examples of such rings include the algebra of endomorphisms on an infinite-dimensional vector space, and the Weyl algebras over fields of characteristic $0$; 
see \cite[Examples~1.1]{Vit21MultilinPolySurj}.

This means that the polynomial $f(a,b,c,d)=[a,b][c,d]$ is surjective on unital rings admitting a surjective inner derivation.
Much more generally, Vitas showed in \cite[Theorem~1.2]{Vit21MultilinPolySurj} that $A = f(A)$ for every nonzero multilinear polynomial~$f$ on a unital algebra $A$ over a field such that $A$ admits a surjective inner derivation.
\end{exa}

Let $S$ be a unital ring and let $n\geq 2$. 
While every matrix of trace zero in $M_n(S)$ is a sum of two commutators by \cite[Theorem~15]{Mes06CommutatorRings}, it was shown by Rosset and Rosset in \cite{RosRos00TraceZeroNotCommutator} that not every matrix of trace zero in $M_n(S)$ is necessarily a commutator. 
In forthcoming work \cite{BreGarThi24pre:ProdCommutatorsMatrixRgs}, we show that the answer to the following question is `Yes'.

\begin{qst}
Do there exist a unital ring $S$ and $n \geq 2$ such that $\xi(M_n(S))=2$?
\end{qst}

\begin{exa}
\label{exa:DivisionRing}
Let $D$ be a division ring.
Since $D$ is simple, we see that $D$ is generated by its commutators as an ideal if and only if $D$ is not commutative (that is, not a field).
Let us therefore consider the case that $D$ is a noncommutative division ring.
We will show that  
\[
D
= [D,D]_1 + [D,D]_1 \multNoSpan [D,D]_1
= [D,D]_1\multNoSpan[D,D]_1 + [D,D]_1\multNoSpan[D,D]_1,
\]
so in particular $\xi(D) \leq 2$.

It is clear that $[D,[D,D]_1]_1 \subseteq [D,D]_1$. 
We claim that also $[D,[D,D]_1]_1 \subseteq [D,D]_1^{\multNoSpan 2}$.
To see this, let $r,s,t \in D$.
If $[s,t] = 0$, then $[r,[s,t]\big] = 0 \in [D,D]_1^{\multNoSpan 2}$.
If $[s,t] \neq 0$, then 
\[
[r,[s,t]\big]
= [r[s,t]^{-1},[s,t]\big][s,t] \in [D,D]_1^{\multNoSpan 2}.
\]

By \autoref{prp:CommutativitySemiprime}, we have $\big[[D,D],[D,D]\big] \neq \{0\}$.
Since every nonzero element in $D$ is invertible, this shows that there exist $a \in D$ and $v,w \in [D,D]_1$ such that $1 = a[v,w]$.
Now, given any $d \in D$, we have 
\[
d
= da[v,w] 
= [dav,w] + [w,da]v 
\in \big[ D, [D,D]_1 \big]_1 + [D,D]_1\multNoSpan[D,D]_1.
\]
By the previous comments, the above is contained both in $[D,D]_1 + [D,D]_1\multNoSpan[D,D]_1$ and in $[D,D]_1\multNoSpan[D,D]_1+[D,D]_1\multNoSpan[D,D]_1$, as desired.
\end{exa}

It was shown by Harris in \cite{Har58CommutatorDivRg} that there exist division rings $D$ with $D = [D,D]_1$, hence also $D=[D,D]_1^{\multNoSpan 2}$.
In that case, we have $\xi(D) = 1$.

\begin{qst}
\label{qst:DivisionRing}
Do we have $\xi(D) = 1$ for every noncommutative division ring $D$?
\end{qst}

\begin{qst}
\label{qst:MatrixDivisionRing}
Do we have $\xi(M_n(D)) = 1$ for every $n \geq 2$ and every (noncommutative) division ring $D$?
\end{qst}

Recall that a ring is said to be \emph{Artinian} if it has no infinite descending sequence of left or right ideals.

\begin{exa}
\label{exa:Semisimple}
Let $R$ be a semisimple Artinian ring.
By the Artin–Wedderburn Theorem, there exist $m,n_1,\ldots,n_m\in\mathbb{N}$
and division rings $D_1,\ldots, D_m$ such that
$R \cong \bigoplus_{j=1}^m M_{n_j}(D_j)$.
Since $M_{n_j}(D_j)$ is simple for all $j=1,\ldots,m$, it follows that $R$ is generated by its commutators as an ideal if and only if for all $j$ with $n_j = 1$, the division ring $D_j$ is not commutative, that is, if and only if no summand $M_{n_j}(D_j)$ is a field.
Let us assume that this is the case.
Applying \autoref{prp:MatrixSizeN} for the summands with $n_j \geq 2$, and \autoref{exa:DivisionRing} for the summands with $n_j = 1$, we deduce that $\xi(R) \leq 2$.

If $R$ is additionally finite-dimensional over an algebraically closed field $\mathbb{F}$, then $R \cong \bigoplus_{j=1}^m M_{n_j}(\mathbb{F})$ for some integers $n_j \geq 1$.
In this case, $R$ is generated by its commutators if and only if $n_j \geq 2$ for every $j$, in which case $\xi(R) = 1$ by \autoref{exa:MatrixAlgebra}.
\end{exa}

We now turn to estimates relating the invariant $\xi$ of a ring and a subring.
The following is a first crude estimate that can probably be improved.

\begin{prp}
\label{prp:EstimateCrude}
Let $R$ be a unital ring, and let $S \subseteq R$ be a subring with $1 \in S$.
Then $\xi(R) \leq 15\xi(S)^3$.
\end{prp}
\begin{proof}
Set $M_0 := [S,S]_1$ and $N := \xi(S)$.
We clearly have $[M_0,S]_1 \subseteq M_0$. 
Moreover, we claim that $[M_0^{\multNoSpan 2},S]_1 \subseteq \sum^2 M_0^{\multNoSpan 2}$. 
To see this, let $s \in S$ and $x,y \in M_0$.
Then
\[
[s,xy] 
= [s,x]y + x[s,y]
\in M_0^{\multNoSpan 2}+M_0^{\multNoSpan 2},
\]
as desired.
Using the above, it follows from \autoref{prp:Ideal-LeftIdeal} that 
\[
\label{eqn:Incl}\tag{5.2}
S \multNoSpan [M_0,M_0^{\multNoSpan 2}]_1 \multNoSpan S 
\subseteq \sum{}^4 S \multNoSpan [M_0,M_0^{\multNoSpan 2}]_1.
\]
Note also that $S=\sum^NM_0^{\multNoSpan 2}$ immediately implies that 
\[\label{eqn:Incl2}\tag{5.3}
S \multNoSpan [S,S]_1
= \sum{}^{N^2} S \multNoSpan [M_0^{\multNoSpan 2},M_0^{\multNoSpan 2}]_1.
\]

Using unitality of $S$ and the identity $[ab,cd]=a[b,cd]+[a,cd]b$ at the fourth step, we get
\begin{align*}\tag{5.4}
S 
&= \sum{}^N [S,S]_1 \multNoSpan [S,S]_1
\subseteq \sum{}^N S \multNoSpan [S,S]_1
\stackrel{\mathrm{(5.3)}}{=} \sum{}^{N^3} S \multNoSpan [M_0^{\multNoSpan 2},M_0^{\multNoSpan 2}]_1 \\
&= \sum{}^{N^3} \left( S \multNoSpan [M_0,M_0^{\multNoSpan 2}]_1 + S \multNoSpan [M_0,M_0^{\multNoSpan 2}]_1 \multNoSpan S \right)
\stackrel{\mathrm{(5.2)}}{\subseteq} \sum{}^{5N^3} S \multNoSpan [M_0,M_0^{\multNoSpan 2}]_1.
\end{align*}

Set $L_0 := [R,R]_1$.
We have $[R,L_0]_1 \subseteq L_0$ and $[L_0^{\multNoSpan 2},R]_1 \subseteq L_0$, and therefore $R \multNoSpan [L_0,L_0^{\multNoSpan 2}]_1 \subseteq \sum^3 L_0^{\multNoSpan 2}$ by \autoref{prp:Ideal-L-L2}.
Using that $M_0\subseteq L_0$ and using the above at the last step, we get 
\[
R 
= R \multNoSpan S
\stackrel{\mathrm{(5.4)}}{\subseteq} \sum{}^{5N^3} R \multNoSpan S \multNoSpan [M_0,M_0^{\multNoSpan 2}]_1
\subseteq \sum{}^{5N^3} R \multNoSpan [L_0,L_0^{\multNoSpan 2}]_1
\subseteq \sum{}^{15N^3} L_0^{\multNoSpan 2}.
\]

Thus, $R = \sum^{15N^3}[R,R]_1\multNoSpan[R,R]_1$ and therefore $\xi(R) \leq 15 \xi(S)^3$.
\end{proof}

If $R$ is a unital ring that contains a unital subring of the form $\bigoplus_{j=1}^m M_{n_j}(S_j)$ for unital rings $S_j$ and $n_j \geq 2$, then \autoref{prp:EstimateCrude} and \autoref{prp:MatrixSizeN} show that $\xi(R) \leq 120$.
We now present a method that can be used to obtain a much better bound, and which we will use in \autoref{prp:XiMatrixSubalgebra} to verify that $\xi(R) \leq 3$.

\begin{lma}
\label{prp:EstimateFine}
Let $R$ be a unital ring, let $S \subseteq R$ be a subring with $1 \in S$, and let $m,n \in \NN$.

\begin{enumerate}
\item
If $1 \in \sum^m S\multNoSpan\big[[S,S]_1,[S,S]_1\big]_1$, then $R = \sum^m \big([R,R]_1+[R,R]_1\multNoSpan[R,R]_1\big)$.
In particular, if there exist $a \in S$ and $u,v \in [S,S]_1$ with $1 = a[u,v]$, then $R = [R,R]_1 + [R,R]_1\multNoSpan[R,R]_1$.
\item
If $1 \in \sum^n S\multNoSpan\big[[S,S]_1,[S,S]_1^{\multNoSpan 2}]$, then $\xi(R)\leq 3n$. 
In particular, if there exist $a \in S$ and $u,v,w \in [S,S]_1$ such that $1 = a[u,vw]$, then $\xi(R)\leq 3$.
\end{enumerate}
\end{lma}
\begin{proof}
Set $L_0 := [R,R]_1$, and note that $[R,L_0]_1,[L_0,R]_1 \subseteq L_0$ and $[L_0^{\multNoSpan 2},R]_1 \subseteq L_0$.

(1). 
Using the assumption at the first step, and \autoref{prp:Ideal-L-L} at the second step, we get
\[
R
= \sum{}^m R\multNoSpan[L_0,L_0]_1
\subseteq \sum{}^m (L_0+L_0^{\multNoSpan 2})
= \sum{}^m ([R,R]_1 + [R,R]_1\multNoSpan[R,R]_1).
\]

(2). 
Using the assumption at the first step, and \autoref{prp:Ideal-L-L2} at the second step, we get
\[
R 
= \sum{}^n R\multNoSpan[L_0,L_0^{\multNoSpan 2}]_1
\subseteq \sum{}^{3n} L_0^{\multNoSpan 2}
= \sum{}^{3n} [R,R]_1\multNoSpan[R,R]_1. \qedhere
\]
\end{proof}

We will need to know that matrix algebras over $\ZZ$ contain certain special commutators.

\begin{lma}
\label{prp:WitnessesIntegralMatrix}
Let $n \geq 2$.
Then the integral matrix ring $M_n(\ZZ)$ contains matrices $u,v,w \in [M_n(\ZZ),M_n(\ZZ)]_1$ such that $[u,v]$ is invertible and $v=vw$.
\end{lma}
\begin{proof}
Suppose $n = 2$.
Set
\[
u := \begin{pmatrix}
0 & 1 \\
0 & 0
\end{pmatrix}
= \left[
\begin{pmatrix}
0 & 1 \\
0 & 0
\end{pmatrix},
\begin{pmatrix}
0 & 0 \\
0 & 1
\end{pmatrix}
\right] \ \ \ \ 
v := \begin{pmatrix}
0 & 0 \\
1 & 0
\end{pmatrix}
= \left[
\begin{pmatrix}
0 & 0 \\
1 & 0
\end{pmatrix},
\begin{pmatrix}
1 & 0 \\
0 & 0
\end{pmatrix}
\right]. 
\]
One checks that
\[
[u,v]
= \left[
\begin{pmatrix}
0 & 1 \\
0 & 0
\end{pmatrix},
\begin{pmatrix}
0 & 0 \\
1 & 0
\end{pmatrix}
\right]
= \begin{pmatrix}
1 & 0 \\
0 & -1
\end{pmatrix},
\]
which is invertible. 
Since one may also readily check that $v=v[u,v]$, we may take $w=[u,v]$.

Suppose $n = 3$.
Consider
\begin{align*}
u &:= \begin{pmatrix}
0 & 0 & 1 \\
1 & 0 & 0 \\
0 & 0 & 0
\end{pmatrix}
= \left[
\begin{pmatrix}
1 & 0 & 0 \\
0 & 2 & 0 \\
0 & 0 & 0
\end{pmatrix},
\begin{pmatrix}
0 & 0 & 1 \\
1 & 0 & 0 \\
0 & 0 & 0
\end{pmatrix}
\right], \\
v &:= \begin{pmatrix}
0 & 0 & 0 \\
1 & 0 & 0 \\
0 & 1 & 0
\end{pmatrix}
= \left[
\begin{pmatrix}
0 & 0 & 0 \\
1 & 0 & 0 \\
0 & 1 & 0
\end{pmatrix},
\begin{pmatrix}
2 & 0 & 0 \\
0 & 1 & 0 \\
0 & 0 & 0
\end{pmatrix}
\right].
\end{align*}
Then
\[
[u,v]
= \left[
\begin{pmatrix}
0 & 0 & 1 \\
1 & 0 & 0 \\
0 & 0 & 0
\end{pmatrix},
\begin{pmatrix}
0 & 0 & 0 \\
1 & 0 & 0 \\
0 & 1 & 0
\end{pmatrix}
\right]
= \begin{pmatrix}
0 & 1 & 0 \\
0 & 0 & -1 \\
-1 & 0 & 0
\end{pmatrix}
\]
is readily seen to be invertible. 
Moreover, setting
\[
w
:= \begin{pmatrix}
1 & 0 & 0 \\
0 & 1 & 0 \\
0 & 0 & -2
\end{pmatrix}
= \left[
\begin{pmatrix}
0 & 1 & 0 \\
0 & 0 & 2 \\
0 & 0 & 0
\end{pmatrix},
\begin{pmatrix}
0 & 0 & 0 \\
1 & 0 & 0 \\
0 & 1 & 0
\end{pmatrix}
\right],
\]
one checks that $vw=v$ as well.

Suppose that $n \geq 4$.
Find $k,\ell\in\mathbb{N}$ such that $n=2k+3\ell$. 
One can find the desired matrices in $M_n(\ZZ)$ as block-diagonal matrices using $k$ blocks in $M_2(\ZZ)$ and $\ell$ blocks in $M_3(\ZZ)$. 
We omit the details.
\end{proof}

\begin{thm}
\label{prp:XiMatrixSubalgebra}
Let $R$ be a unital ring, and suppose that there exist unital rings $S_1,\ldots,S_m$ and natural numbers $n_1,\ldots,n_m \geq 2$ such that $\bigoplus_{j=1}^m M_{n_j}(S_j)$
embeds unitally into $R$.
Then
\[
R = [R,R]_1 + [R,R]_1\multNoSpan[R,R]_1, \andSep 
R = \sum{}^3 [R,R]_1\multNoSpan[R,R]_1.
\]
In particular, $\xi(R) \leq 3$.
\end{thm}
\begin{proof}
Set $S := \bigoplus_{j=1}^m M_{n_j}(S_j)$. Without loss of 
generality, we assume that $S$ is a subring of $R$.
Applying \autoref{prp:WitnessesIntegralMatrix} in each summand $M_{n_j}(S_j)$ and adding them, we find $a \in S$ and $u,v,w \in [S,S]_1$ such that $1 = a[u,v] = a[u,vw]$.
The result now follows from \autoref{prp:EstimateFine}.
\end{proof}

\begin{qst}
In the setting of \autoref{prp:XiMatrixSubalgebra}, does one have $\xi(R) \leq 2$?
\end{qst}

\begin{rmk}
\autoref{prp:XiMatrixSubalgebra} applies in particular to matrix rings and shows that for every unital ring $S$ and $n \geq 2$, every element in $M_n(S)$ is a sum of a commutator with a product of two commutators.
Similarly, one can show that $R = [R,R]_1 + [R,R]_1\multNoSpan[R,R]_1$ whenever $R$ is a noncommutative division ring, or a semisimple, Artinian ring without commutative summands.
\end{rmk}

\section{Bounds on the number of commutators for \texorpdfstring{$C^*$}{C*}-algebras}
\label{sec:BoundsCAlg}

In this section, we prove estimates for the invariant $\xi(A)$ in the case that $A$ is a unital \ca{}.
We show that $\xi(A) \leq 3$ if $A$ is properly infinite (\autoref{exa:ProperlyInfinite}), or has real rank zero (\autoref{prp:RR0}).
We prove the estimate $\xi(A) \leq 6$ whenever the Jiang-Su algebra $\mathcal{Z}$ embeds unitally into $A$, in particular for all $\mathcal{Z}$-stable \ca{s};
see \autoref{prp:DimensionDrop}.

\medskip

A unital \ca\ $A$ is said to be \emph{properly infinite} if there exist mutually orthogonal projections $p,q\in A$ and isometries $s,t\in A$ satisfying $ss^*=p$ and $tt^*=q$.
Equivalently, and with $\mathcal{O}_\infty$ denoting the infinite Cuntz algebra, there is a unital embedding $\mathcal{O}_\infty \to A$;
see \cite[Proposition~III.1.3.3]{Bla06OpAlgs}.

\begin{exa}
\label{exa:ProperlyInfinite}
Let $A$ be a unital, properly infinite \ca. 
Then
\[
A = \sum{}^2 [A,A]_1
= [A,A]_1 + [A,A]_1\multNoSpan[A,A]_1
= \sum{}^3 [A,A]_1\multNoSpan[A,A]_1,
\]
so in particular $\xi(A) \leq 3$.
Indeed, the first equality was shown by Pop; 
see \cite[Remark~3]{Pop02FiniteSumsCommutators}.
The two last equalities follow from \autoref{prp:XiMatrixSubalgebra} using that $A$ admits a unital embedding of the Cuntz algebra $\mathcal{O}_\infty$, which in turn admits a unital embedding of $M_2(\CC) \oplus M_3(\CC)$.
\end{exa}

It would be interesting to compute $\xi(A)$ for some particular cases of properly (or even purely) infinite \ca{s}. 
For example:

\begin{qst} 
What is $\xi(\mathcal{O}_\infty)$?
\end{qst}

For Cuntz algebras $\mathcal{O}_n$ with finite $n$, a smaller upper bound can be given.
Indeed, for $n\geq 2$ it is well-known that $\mathcal{O}_n$ contains a unital copy of $M_n(\CC)$.
(Indeed, if $s_1,\ldots,s_n$ denote the canonical isometries generating $\mathcal{O}_n$, one 
can readily check that $\{s_js_k^*\colon j,k=1,\ldots,n\}$ is a system of matrix units, so it generates a (unital) copy of $M_n(\CC)$.)
It follows that $\mathcal{O}_n$ is isomorphic to the ring of $n$-by-$n$ matrices over a unital ring, and therefore $\xi(\mathcal{O}_n) \leq 2$ by \autoref{prp:MatrixSizeN}. 

The same argument applies to the $L^p$-versions $\mathcal{O}_n^p$ of the Cuntz algebras 
introduced in \cite{Phi12arX:LpAnalogsCtz} and studied in \cite{GarLup17ReprGrpdLp, ChoGarThi19arX:LpRigidity}, thus giving $\xi(\mathcal{O}_n^p) \leq 2$ for every $p\in [1,\infty)$ and every natural number $n \geq 2$.
We do not know whether $\xi(\mathcal{O}_n^p)=1$.

\medskip

A unital \ca{} is said to have \emph{real rank zero} if the self-adjoint elements with finite spectrum are dense among all self-adjoint elements.
See \cite[Section~V.3.2]{Bla06OpAlgs} for details.

\begin{thm}
\label{prp:RR0}
Let $A$ be a unital \ca{} of real rank zero that admits no characters.
Then $\xi(A)\leq 3$ and 
\[
A 
= [A,A]_1 + [A,A]_1\multNoSpan[A,A]_1.
\]
\end{thm}
\begin{proof}
Set $B := M_2(\CC) \oplus M_3(\CC)$.
By \cite[Proposition~5.7]{PerRor04AFembeddings}, the unit of $A$ is weakly divisible of degree~$2$, that is, there exists a unital homomorphism $B \to A$;
see \cite[Page~164, Lines 10ff]{PerRor04AFembeddings}.
It follows that $A$ contains a unital copy of $B$, or one of its unital quotients $M_2(\CC)$ or $M_3(\CC)$.
Now the result follows from \autoref{prp:XiMatrixSubalgebra}.
\end{proof}

We now turn to \ca{s} that admit an embedding of the Jiang-Su algebra~$\mathcal{Z}$, which is a unital, simple, separable, nuclear \ca{} with unique tracial state and without nontrivial idempotents;
see \cite[Example~3.4.5]{Ror02Classification}.
There are plenty of \ca s that contain $\mathcal{Z}$ unitally but which do not contain any matrix subalgebras, so that \autoref{prp:XiMatrixSubalgebra} does not apply to them. One such example is given
by the reduced group \ca{} $C^*_\lambda(\mathbb{F}_n)$ of the free group; see \autoref{eg:CFn}.

We will need the so-called \emph{dimension drop algebra} $Z_{2,3}$,
which is given by
\[
Z_{2,3}
=\left\{ f\colon [0,1]\to M_2(\CC) \otimes M_3(\CC) \mbox{ continuous}\colon 
 \begin{aligned}
     f(0)\in M_2(\CC) \otimes 1,\\ f(1)\in 1 \otimes M_3(\CC)
     \end{aligned}
     \right\}.
\]
It is known that $\mathcal{Z}$ contains $Z_{2,3}$ as a unital subalgebra.

A unital \ca{} is said to be \emph{$\mathcal{Z}$-stable} if $A$ is $\ast$-isomorphic to the \ca{ic} tensor product $A \otimes \mathcal{Z}$.
Such algebras clearly admit a unital embedding of $\mathcal{Z}$, and therefore of $Z_{2,3}$.

The Jiang-Su algebra and $\mathcal{Z}$-stable \ca{s} play a crucial role in Elliott's classification program of simple, nuclear \ca{s}.
For a general overview, as well as the statement of the recent classification theorem and corresponding references, we refer the reader to Winter's ICM proceedings \cite{Win18ICM}.

\begin{thm}
\label{prp:DimensionDrop}
Let $A$ be a unital \ca{} that admits a unital $\ast$-homomorphism $Z_{2,3} \to A$.
Then $\xi(A)\leq 6$.
This applies in particular to all unital, $\mathcal{Z}$-stable \ca{s}.
\end{thm}
\begin{proof}
We claim that there exist $a,b\in Z_{2,3}$ and $q,r,s,x,y,z \in [Z_{2,3},Z_{2,3}]_1$ such that
\[
1 = a[q,rs]+b[x,yz].
\]
Once we prove this, the fact that $\xi(A)\leq 6$ will follow from part~(2) of \autoref{prp:EstimateFine} with $n=2$.

Use \autoref{prp:WitnessesIntegralMatrix} with $n=3$ to find $u,v,w \in [M_3(\ZZ),M_3(\ZZ)]_1$ and $e \in M_3(\ZZ)$ satisfying $1_3=e[u,vw]$.
Define functions $a,q,r,s \colon [0,1] \to M_2(\CC) \otimes M_3(\CC)$ by
\[
a(t) = t^{\frac{1}{2}} (1_2 \otimes e), \quad
q(t) = t^{\frac{1}{4}} (1_2 \otimes u), \quad
r(t) = t^{\frac{1}{8}} (1_2 \otimes v), \andSep 
s(t) = t^{\frac{1}{8}} (1_2 \otimes w). 
\]
Denote by $h \in Z_{2,3}$ the function given by $h(t)=t1_{6}$. 
One readily checks that $a,q,r,s$ belong to $Z_{2,3}$; 
that $q,r,s$ are commutators in $Z_{2,3}$, and that $h = a[q,rs]$.

Similarly, using \autoref{prp:WitnessesIntegralMatrix} with $n=2$, we find $u',v',w' \in [M_2(\ZZ),M_2(\ZZ)]_1$ and $e' \in M_2(\ZZ)$ satisfying $1_2 = e'[u',v'w']$. 
Define $b,x,y,z \colon [0,1] \to  M_2(\CC) \otimes M_3(\CC)$ by
\begin{align*}
b(t) &= (1-t)^{\frac{1}{2}} (e' \otimes 1_3), \quad
x(t) = (1-t)^{\frac{1}{4}} (u' \otimes 1_3), \\
y(t) &= (1-t)^{\frac{1}{8}} (v' \otimes 1_3), \andSep
z(t) = (1-t)^{\frac{1}{8}} (w' \otimes 1_3). 
\end{align*}
One readily checks that $b,x,y,z$ belong to $Z_{2,3}$;
that $x,y,z$ are commutators in~$Z_{2,3}$, and that $1-h =b[x,yz]$.
It follows that $1=a[q,rs]+b[x,yz]$, as desired.
\end{proof}

Although the primary interest in the theorem above is to $\mathcal{Z}$-stable \ca{s}, it also gives information for certain reduced group \ca{s}:

\begin{exa}
\label{eg:CFn}
Let $G$ be a discrete group containing a nonabelian free group. 
Then $\xi(C^*_\lambda(G)) \leq 6$.
Indeed, it was shown in \cite[Proposition~4.2]{ThiWin14GenZStableCa} that $C^*_\lambda(\mathbb{F}_2)$ contains a unital copy of $\mathcal{Z}$, and the assumptions on $G$ guarantee that $C^*_\lambda(\mathbb{F}_2)$ embeds unitally into $C^*_\lambda(G)$. 
Now the assertion follows from \autoref{prp:DimensionDrop}. 

The \ca s here considered often contain no nontrivial idempotents, and hence no matrix subalgebras, for example for $G=\mathbb{F}_n$. In such situations,
\autoref{prp:XiMatrixSubalgebra} is not applicable.
\end{exa}

The methods we have developed to show that $\xi(C^*_\lambda(\mathbb{F}_n)) \leq 6$ do not seem to give any
information about the $L^p$-versions $F^p_\lambda(\mathbb{F}_n)$ 
of these group algebras introduced by Herz in \cite{Her73SynthSubgps} and further studied, among others, in \cite{GarThi15GpAlgLp, GarThi19ReprConvLq, GarThi22IsoConv}. 
It would be interesting to find explicit upper bounds for $\xi(F^p_\lambda(\mathbb{F}_n))$.

\section{Outlook}
\label{sec:Outlook}

We end this paper with some questions for future work.
Given a \ca{} $A$, consider the following properties:
\begin{enumerate}
\item
We have $A = A[A,A]A$.
\item
We have $A = [A,A]^2$.
\item
We have $[A,A] = \big[[A,A],[A,A]^2\big]$.
\item
We have $A[A,A]A = [A,A]^2$.
\item
We have $[A,A] \subseteq [A,A]^2$.
\end{enumerate}
Then the implications `(1)$\Leftrightarrow$(2)$\Rightarrow$(3)$\Rightarrow$(4)$\Rightarrow$(5)' hold.
Indeed, (1) and~(2) are equivalent by \autoref{prp:CommutatorsCAlg}.
To show that~(2) implies~(3), assume that $A=[A,A]^2$.
Given $a,b,c \in A$, we have $[ab,c] = [a,bc] + [b,ca]$ and thus
\begin{align*}
[A,A]
&= \big[[A,A]^2,A\big]
\subseteq \big[[A,A],[A,A]A\big] + \big[[A,A],A[A,A]\big] \\
&\subseteq \big[[A,A],A\big]
\subseteq \big[[A,A],[A,A]^2\big].
\end{align*}

To show that~(3) implies~(4), first note that the inclusion  $[A,A]^2 \subseteq A[A,A]A$ holds in general, since the ideal $A[A,A]A$ contains $[A.A]$ by \cite[Proposition~3.2]{GarThi23arX:PrimeIdealsCAlg}.
Further, assuming that $[A,A] = [[A,A],[A,A]^2]$, using \autoref{prp:Ideal-L-L2} at the second step, we get
\[
A[A,A]A
\ = \ A\big[[A,A],[A,A]^2\big]A 
\ \subseteq \ [A,A]^2.
\]
Finally, we see that~(4) implies~(5), again using that $[A,A] \subseteq A[A,A]A$.

If $A$ is commutative, then (2) does not hold, while~(3) is satisfied.
Thus, (3) need not imply~(2), but it remains unclear if~(3)-(5) are equivalent.
In fact, it is possible that~(3)-(5) are always true.

\begin{qst}
\label{qst:Commutators}
Let $A$ be a \ca.

(a) Do we have $[A,A] = \big[[A,A],[A,A]^2\big]$?

(b) Do we have $A[A,A]A = [A,A]^2$?

(c) Do we have $[A,A] \subseteq [A,A]^2$?
\end{qst}

For \ca{s} that are generated by their commutators as an ideal (or more generally, \ca{s} for which the commutator ideal is closed), all subquestions  of \autoref{qst:Commutators} have a positive answer.
Further, by passing to the closed ideal generated by commutators, one may assume that the \ca{} has no character.
Thus, \autoref{qst:Commutators} is only unclear for \ca{s} that have no character but are not generated by their commutators as an ideal, such as the one described in \autoref{exa:NoCharNotGenCommutators}. 

\begin{rmk}
Assume that a \ca{} $A$ satisfies $[A,A] \subseteq [A,A]^2$.
We claim that the commutator ideal $I := \widetilde{A}[A,A]\widetilde{A}$ is semiprime, that is, an intersection of prime ideals.
To see this, we note that $I$ is idempotent, that is, $I=I^2$.
Indeed, the inclusion $I^2 \subseteq I$ is clear, and conversely we have
\[
I 
= \widetilde{A}[A,A]\widetilde{A}
\subseteq \widetilde{A}[A,A]^2\widetilde{A}
\subseteq \widetilde{A}[A,A]\widetilde{A}[\widetilde{A}A,A]\widetilde{A}
= I^2.
\]
Since an ideal in a \ca{} is idempotent if and only if it is semiprime
by \cite[Theorem~A]{GarKitThi23arX:SemiprimeIdls}, the claim follows.
\end{rmk}


\providecommand{\bysame}{\leavevmode\hbox to3em{\hrulefill}\thinspace}

\end{document}